\documentclass[a4paper]{amsart}

\usepackage{lmodern,microtype}
\usepackage{amsmath,amstext,amssymb,mathrsfs,amscd,amsthm,indentfirst}
\usepackage{amsfonts}
\usepackage[dvipsnames,svgnames,x11names,hyperref]{xcolor}

\usepackage[pagebackref,colorlinks,citecolor=Mahogany,linkcolor=Mahogany,urlcolor=Mahogany,filecolor=Mahogany]{hyperref}

\usepackage{booktabs}
\usepackage{verbatim}
\usepackage{enumerate}
\usepackage{graphicx}

\usepackage[utf8]{inputenc}
\usepackage{textcomp}
\usepackage{color}
\usepackage{mathtools}
\usepackage{multirow}
\usepackage{blindtext}
\usepackage{enumitem}
\usepackage{subcaption}
\usepackage{wrapfig}
\usepackage{float}
\usepackage{pst-node}
\usepackage{mathdots}
\usepackage{scalerel}
\usepackage{stackengine,wasysym}

\newcommand\reallywidetilde[1]{\ThisStyle{%
  \setbox0=\hbox{$\SavedStyle#1$}%
  \stackengine{-.1\LMpt}{$\SavedStyle#1$}{%
    \stretchto{\scaleto{\SavedStyle\mkern.2mu\AC}{.5150\wd0}}{.6\ht0}%
  }{O}{c}{F}{T}{S}%
}}

\usepackage{tikz, tikz-cd}
\usetikzlibrary{matrix,calc,positioning,arrows,decorations.pathreplacing,decorations.markings,patterns}
\tikzset{->-/.style={decoration={
			markings,
			mark=at position #1 with {\arrow{>}}},postaction={decorate}}}
\tikzset{-<-/.style={decoration={
					markings,
					mark=at position #1 with {\arrow{<}}},postaction={decorate}}}

\DeclareMathOperator{\Sym}{Sym}

\DeclareMathOperator{\id}{id}

\DeclareMathOperator{\Z}{\mathbb{Z}}
\DeclareMathOperator{\A}{\mathcal{A}}

\DeclareMathOperator{\R}{\mathbb{R}}

\DeclareMathOperator{\SP}{SP}
\DeclareMathOperator{\Th}{Th}

\DeclareMathOperator{\PL}{PL}
\DeclareMathOperator{\bbD}{\mathbb{D}}
\DeclareMathOperator{\bbS}{\mathbb{S}}

\mathchardef\ordinarycolon\mathcode`\:
\mathcode`\:=\string"8000
\begingroup \catcode`\:=\active
  \gdef:{\mathrel{\mathop\ordinarycolon}}
\endgroup

\usepackage{amsthm}
\theoremstyle{plain}
\newtheorem*{theorem*}{Theorem}
\newtheorem{athm}{Theorem}
    
\newtheorem{theorem}{Theorem}[section]

\newtheorem{proposition}[theorem]{Proposition}
\newtheorem*{question*}{Question}

\newtheorem{lemma}[theorem]{Lemma}
\newtheorem*{lemma*}{Lemma}

\newtheorem{corollary}[theorem]{Corollary}
\newtheorem*{corollary*}{Corollary}

\theoremstyle{definition}

\newtheorem{example}[theorem]{Example}

\theoremstyle{remark}

\newtheorem*{remark*}{Remark}

\numberwithin{equation}{section}

\author{Andreas Stavrou}
\address{Department of Mathematics,
University of Chicago,
5734 S. University Avenue,
Chicago, IL, 60637, United States of America}
\email{andreasstavrou@uchicago.edu}

\title{Superposition of configurations and scanning}

\date{\today}

\begin{document}

\maketitle
\begin{abstract}
We endow the cohomology of configuration spaces of a manifold with a product arising from superposing configurations. We prove that, under the scanning isomorphism, this product corresponds to the cup-product of the section space of the standard scanning bundle of the manifold.
\end{abstract}

\section*{Introduction}
Let $M$ be the interior of a  compact connected $d$-manifold, possibly with boundary. We are interested in the \textit{(unordered) configuration spaces} 
$$C_n(M):=\{(m_1,...,m_n)\in {M}^n:m_i\neq m_j \text{ if } i\neq j\}/\mathfrak{S}_n,$$
for $n\ge 0$,
where the symmetric group $\mathfrak{S}_n$ acts by permuting the coordinates.

The homology of $C_n(M)$ was classically studied using the \textit{scanning map}, originally due to \cite{McDuff1975,Segal73}, whose following incarnation is from \cite{bodigheimer88rationalcohomologyofsurfaces,BODIGHEIMER1989111generalconfigurations}. 
Let $T^+M$ be the fibrewise one-point compactification of the tangent bundle $TM$ of $M$, and let $T^+M\wedge_fS^{2l}$ be the fibrewise smash of $T^+M$ with $S^{2l}$, so that $T^+M\wedge_fS^{2l}$ is an $S^{d+2l}$-bundle over $M$. 
Denote by $\Gamma_{\partial}(M,S^{2l})$ the space of its sections over $M$ that map the boundary $\partial M$ constantly to the basepoint of the fibre.
\begin{athm}[\cite{bodigheimer88rationalcohomologyofsurfaces,BODIGHEIMER1989111generalconfigurations}]\label{thm:A}
For $l\ge 1$, there is an isomorphism of graded abelian groups
\begin{equation*}\label{eq:introscanningisomorphism}
s:{H}^*\left(\Gamma_{\partial}(M,S^{2l})\right)\longrightarrow \bigoplus_{n\ge 0}H^*(C_n(M))[2ln].
\end{equation*}
\end{athm}

The domain of $s$ is naturally a graded ring with the cup-product; our aim is to exhibit a meaningful interpretation of this ring structure on the target of $s$. Consider the coloured configuration space $C_{n,m}(M)$ of $n$ blue points and $m$ red points, all distinct. There are maps 
\begin{center}
    \begin{tikzcd}
        C_n(M)\times C_m(M) & C_{n,m}(M)\lar[hook,"i"']\rar["p"] & C_{n+m}(M),
    \end{tikzcd}
\end{center}
where $i$ is the open inclusion, and $p$ is the ${n+m \choose m}$-sheeted covering map which forgets the colours.
Combining K\"unneth, the cohomology map $i^*$ and the transfer map $p_!$, we obtain the \textit{superposition product} $${\sup}_{n,m}:=p_!\circ i^*:H^*(C_n(M))\otimes H^*(C_m(M))\longrightarrow H^*(C_{n+m}(M)),$$
for $n,m\ge 0$. In Section \ref{sec:alternativedefinition}, we give an alternative but equivalent definition which justifies the name ``superposition'' and appears in \cite{Moriyama,bianchi2021mapping,randalwilliams2023configuration}.

\begin{athm}[Theorem \ref{thm:main}]\label{thm:B}
The isomorphism $s$ is multiplicative with the cup-product on ${H}^*(\Gamma_{\partial}(M,S^{2l}))$ and the superposition product on $\bigoplus_{n\ge 0}H^*(C_n(M))[2ln]$.
\end{athm}

To prove Theorem \ref{thm:B}, we analyse the topological definition of the map $s$ following \cite{Tillmanthorpe2014} which involves symmetric products. In Section \ref{sec:symmetricproductsetc}, we set up the relevant preliminaries on symmetric products and multivalued functions.

\subsection*{Acknowledgements} The author would like to thank his PhD advisor Oscar Randal-Williams for the conversations which led to this paper. The paper was written while the author was affiliated with the University of Cambridge and funded by the Engineering and Physical Sciences Research Council (project reference: 2261124).

\section{Symmetric products and multivalued functions}\label{sec:symmetricproductsetc}

Given a based space $(X,e)$, the \textit{symmetric product} $\SP(X)$ is the free unital abelian monoid on $X$ with identity $e$, topologised as the quotient
\begin{equation*}
    \SP(X):=\Big(\bigsqcup_{n\ge 1} X^n/\mathfrak{S}_n \Big)\Big/\sim
\end{equation*}
where $\sim$ identifies $(x_1,...,x_n,e)\sim (x_1,...,x_n)$. In other words, elements of $\SP(X)$ are formal sums $\sum_{i\in I} x_i$ with $x_i\in X$ and $I$ a finite indexing set. If $X$ is an unbased space, we set $\SP^+(X)=\SP(X\sqcup *)$, with $*$ the basepoint.

A \textit{multivalued function} of based spaces is a based map $(A,*)\to (\SP(B),e)$. A multivalued function of unbased spaces is a map $A\to \SP^+(B)$.

\subsection{Cohomology maps}\label{sec:DoldThomCohomology}
We explain how multivalued functions induce cohomology maps\footnote{This trick was suggested by Oscar Randal-Williams.}. In literature, the Dold--Thom isomorphism \cite{DoldThom} is usually used to define homology maps instead.

We fix the abelian monoid $\SP(S^n)$ as our model of $K(\Z,n)$ (it is indeed a model by the Dold--Thom theorem). Denoting by $[X,Y]_*$ the set of based maps $X\to Y$ up to based homotopy, we identify the functors of based spaces $\widetilde{H}^n(-)=[-,\SP(S^n)]_*$.
Every map $\alpha:Y\to \SP(S^n)$, representing a cohomology class $[\alpha]\in\widetilde{H}^n(Y)$, factors uniquely through 
\begin{center}
    \begin{tikzcd}
        Y\rar[hook, "\iota_Y"]& \SP(Y)\rar["\widetilde{\alpha}"] & \SP(S^n), 
    \end{tikzcd}
\end{center}
where $\iota_Y$ is the natural inclusion of $Y$ into $\SP(Y)$ and $\widetilde{\alpha}$ is a continuous homomorphism. Now, given a map $f:X\to \SP(Y)$,
we define $$f^*:\widetilde{H}^*(Y)\to \widetilde{H}^*(X)$$ by mapping $[\alpha]\mapsto [\widetilde{\alpha}\circ f]\in [X,K(\Z,n)]_*$.
\begin{proposition}[Naturality]\label{prop:fstarnaturality}
    Let $h:(Y_1,y_1)\to (Y_2,y_2)$ be a based map so that the diagram of based maps
    \begin{equation}\label{eq:naturalityoffstar}
        \begin{tikzcd}
            (X_1,x_1)\dar["g"]\rar["f_1"]& (\SP(Y_1),y_1)\dar["\SP(h)"]\\
            (X_2,x_2)\rar["f_2"] & (\SP(Y_2),y_2)
        \end{tikzcd}
    \end{equation}
    commutes. Then $f_1^*\circ h^*=g^*\circ f_2^*$.
\end{proposition}
\begin{proof}
    Let $\alpha:(Y_2,y_2)\to (\SP(S^n),e)$ represent a class in $\widetilde{H}^*(Y_2)$. Then we have $(f_1^*\circ h^*)([\alpha])=[\widetilde{\alpha\circ h}\circ f_1]$ and $$(g^*\circ f_2^*)([\alpha])=[\widetilde{\alpha}\circ f_2\circ g]=[\widetilde{\alpha}\circ\SP(h)\circ f_1]$$ where the last equality is by the commutativity of diagram \eqref{eq:naturalityoffstar}. Now it suffices to see that the maps $\widetilde{\alpha\circ h}$ and $\widetilde{\alpha}\circ\SP(h)$ are equal: indeed, both maps are homomorphisms restricting to the same map $\alpha\circ h$ on the generating set $Y_1$ of $\SP(Y_1)$.
\end{proof}

It is easy to check that if $f$ is really a single-valued map $f:X\to Y$, then this definition of $f^*$ agrees with the standard definition.

\begin{example}\label{rem:coveringsandtransfers}
By ``inverting'' a finite-sheeted covering map $p:\widetilde{X}\to X$, we obtain a multivalued function $p^{-1}:X\to \SP^+(\widetilde{X})$ mapping
$${x}\longmapsto \sum_{\widetilde{x}\in p^{-1}({x})}\widetilde{x}.$$
Similarly, if $p:(\widetilde{X},e)\to (\SP(X),e)$ is a covering space away from the basepoint, we get a based multivalued function $p^{-1}:(X,e)\to (\SP(\widetilde{X}),e)$. The cohomology map $(p^{-1})^*$ coincides with the transfer maps $p_!$ in both the based and unbased case.
\end{example}

\subsection{Products of maps}
Given two based multivalued functions $f:A\to \SP(B)$ and $g:C\to \SP(D)$, there is a multivalued function $f\triangle g:A\wedge C\to \SP(B\wedge D)$ given by
$$(f\triangle g)(a,c)=\sum_{(i,j)\in I\times J} (a_i,c_j)$$
where $f(a)=\sum_I a_i$ and $g(c)=\sum c_i$. 
More precisely, $f\triangle g$ is the composition of the tautological continuous map \begin{equation*}
\iota_{X,Y}:\SP(X)\wedge \SP(Y)\to \SP(X\wedge Y),
\end{equation*}
and  
the continuous maps $f\wedge g:A\wedge C\to \SP(X)\wedge \SP(Y)$.

Under the identification $\widetilde{H}^n(-)=[-,\SP(S^n)]_*$, the cross product $\times$ is the composition 
\begin{center}
    \begin{tikzcd}
        \left [X,\SP(S
^m)\right ]_*\times \left [Y,\SP(S
^n)\right]_*\dar["-\wedge -"] \\
        \left[X\wedge Y,\SP(S
^m)\wedge \SP(S
^n)\right]_*\dar["\iota_{m,n}\circ -"] \\
        \left[X\wedge Y,\SP(S
^{m+n})\right]_*,
    \end{tikzcd}
\end{center}
where $\iota_{m,n}:=\iota_{S^m,S^n}:\SP(S
^m)\wedge \SP(S
^n)\to \SP(S
^{m+n})$ represents a generator of $H^{m+n}(\SP(S
^m)\wedge \SP(S
^n))\cong \Z$.
We thus obtain a diagram
\begin{equation}\label{eq:diagramcohomologycrosscommutes}
     \begin{tikzcd}
    \widetilde{H}^*(C)\otimes \widetilde{H}^*(D) \rar["f^*\otimes g^*"]\dar["\times"]& \widetilde{H}^*(A)\otimes \widetilde{H}^*(B) \dar["\times"]\\
    \widetilde{H}^*(C\wedge D)\rar["(f\triangle g)^*"] &
    \widetilde{H}^*(A\wedge B).
\end{tikzcd} 
\end{equation}
\begin{proposition}\label{prop:tensorandcohomologymapscommute}
    The diagram \eqref{eq:diagramcohomologycrosscommutes} commutes.
\end{proposition}
\begin{proof}
    We evaluate the two directions of the diagram on a representative $\alpha\otimes \beta$ for a class in $\widetilde{H}^*(C)\otimes \widetilde{H}^*(D)$. Going right-then-down, we have
    \begin{align*}
        \alpha\otimes \beta &\mapsto (\widetilde{\alpha}\circ f )\otimes (\widetilde{\beta}\circ g)\\
        &\mapsto \iota_{n,m}\circ \big((\widetilde{\alpha}\circ f)\wedge (\widetilde{\beta}\circ g)\big)=\iota_{n,m}\circ (\widetilde{\alpha}\wedge \widetilde{\beta})\circ (f\wedge g),
    \end{align*}
    whereas going down-then-right,
        \begin{align*}
        \alpha\otimes \beta &\mapsto \iota_{n,m}\circ ({\alpha}\wedge {\beta})\\
        & \mapsto \reallywidetilde{\iota_{n,m}\circ ({\alpha}\wedge {\beta})}\circ (f\triangle g).
    \end{align*}
    These are equal because the diagram 
    \begin{center}
        \begin{tikzcd}
             & C\wedge D\rar["\alpha \wedge \beta"]\dar["\iota_C\wedge\iota_D"] & \SP(S^n)\wedge \SP(S^m)\rar["\iota_{n,m}"]& \SP(S^{n}\wedge S^m)\\
             &\SP(C)\wedge \SP(D)\arrow[ur,"\widetilde{\alpha}\wedge \widetilde{\beta}"]\dar["\iota_{C,D}"] & &\\
             A\wedge B\rar["f\triangle g"]\arrow[ur,"f\wedge g"]& \SP(C\wedge D)\arrow[rruu,"\reallywidetilde{\iota_{n,m}\circ ({\alpha}\wedge {\beta})}"'] & & 
        \end{tikzcd}
    \end{center}
    commutes: every tilded map is the unique homomorphic extension of its untilded counterpart.
\end{proof}

\subsection{Suspensions and cofibrations}
Let $\Sigma X=S^1\wedge X$ be the \textit{reduced suspension} of a based space $X$. A based multivalued function $f:(X,x)\to (\SP(Y),y)$ induces a suspension map $\Sigma f:=\id_{S^1}\triangle f:(\Sigma X,\Sigma x)\to (\SP(\Sigma Y),\Sigma y)$.
\begin{corollary}\label{cor:fstarwithsuspensions}
    The suspended map $\Sigma f$ commutes with the suspension isomorphism, that is, if $s_Z:\widetilde{H}^*(Z)\to \widetilde{H}^{*+1}(\Sigma Z)$ is the suspension of space $Z$, then $s_X\circ (\Sigma f)^*=f^*\circ s_Y$. 
\end{corollary}
\begin{proof}
    Apply Proposition \ref{prop:tensorandcohomologymapscommute} to $C=A=S^1$ and $f=\id_{S^1}$. The cross-product maps are the suspension isomorphisms in cohomology.
\end{proof}

\begin{lemma}\label{lem:fstarandLES}
    Let \begin{tikzcd}
        A_i\rar["f_i"]& B_i\rar["g_i"] & C_i
    \end{tikzcd} be based cofibration sequences for $i=1,2$, and suppose the diagram of based maps 
    \begin{equation*}
        \begin{tikzcd}
            A_1\rar["f_1"]\dar["a"]& B_1\rar["g_1"]\dar["b"] & C_1\dar["c"]\\
            \SP(A_2)\rar["\SP(f_2)"]& \SP(B_2)\rar["\SP(g_2)"]&\SP(C_2)
        \end{tikzcd}
    \end{equation*}
    commutes. Then the induced cohomology maps $a^*,b^*,c^*$ commute with the long exact sequences of the two cofibration sequences, that is the diagram
    \begin{equation*}
        \begin{tikzcd}
        ...\rar["\partial^*_2"]& \widetilde{H}^*(C_2)\rar["g_2^*"]\dar["c^*"] & \widetilde{H}^*(B_2)\rar["f_2^*"]\dar["b^*"]  & \widetilde{H}^*(A_2)\dar["a^*"] \rar["\partial_2^*"] & \widetilde{H}^{*+1}(C_2)\rar["g_2^*"]\dar["c^*"]& ... \\
        ...\rar["\partial^*_1"]& \widetilde{H}^*(C_1)\rar["g_1^*"] & \widetilde{H}^*(B_1)\rar["f_1^*"] & \widetilde{H}^*(A_1)\rar["\partial_1^*"] & \widetilde{H}^{*+1}(C_1)\rar["g_1^*"]&...
    \end{tikzcd}
    \end{equation*}
    commutes.
\end{lemma}
\begin{proof}
    The left two squares commute by Proposition \ref{prop:fstarnaturality}, so we focus on the rightmost square. Let us recall the construction of the connecting homomorphism $\partial_i^*$ using the reduced mapping cone $\widetilde{C}A_i$ of $A_i$. Since $f_i$ are cofibrations, we have $C_i=B_i/A_i$ and we have the sequence of maps \begin{tikzcd} B_i/A_i&\widetilde{C}A_i\cup_{A_i}B_i \lar[hook,"\sim"', "\iota"]\rar[two heads, "k"]& \Sigma A_i\end{tikzcd}
    where $\iota$ collapses the cone and is a homotopy equivalence, whereas $k$ collapses the subspace $B$. The connecting homomorphism $\partial_i^*$ is the composition $(\iota^*)^{-1}\circ k^*\circ s_{A_i}$. But all three composed maps commute with the induced cohomology maps by the naturality of this topological construction, by Proposition \ref{prop:fstarnaturality}, and by Corollary \ref{cor:fstarwithsuspensions}.
\end{proof}

\section{Scanning}\label{sec:scanningrevisited}

We recall the set-up from \cite{Tillmanthorpe2014} for the space-level model of the scanning isomorphism $s$. Let $(M,M_0)$ be a manifold submanifold pair, and $(X,x_0)$ a based space. The labelled configuration space $C(M,M_0;X)$ is defined as 
$$C(M,M_0;X)=\big\{\{(m_i,x_i)\}_{i\in I}: I \text{ finite set }, m_i\in M,x_i\in X, m_i\neq m_j \text{ if } i\neq j\big\}/\sim,$$
where $\sim$ is the relation generated by $\xi_{I}\sim \xi_I\cup \{(m,x)\}$ if $m\in M_0$ or $x=x_0$. We will denote a generic element of $C(M,M_0;X)$ by $\xi_I$, and given a subset $J\subset I$, we will write $\xi_J:=\{(m_i,x_i)\}_{i\in J}$. 

The space $C:=C(M,M_0;X)$ has a filtration $C_1\subset ... \subset C_k\subset C_{k+1}\subset ...\subset C$ where $C_k$ contains labelled configurations of at most $k$ points. The filtration quotients $D_k=C_k/C_{k-1}$ are naturally based at $\infty$, corresponding to when $(m,x)\in \xi_I$ with $m\in M_0$ or $x=x_0$. Define the wedge sum $V=\bigvee_{k\ge 1}D_k$. The \textit{scanning map} is
\begin{align*}
    \sigma: C&\longrightarrow \SP(V)\\
    \xi_I&\longmapsto \sum_{J\subseteq I} \xi_J.
\end{align*}
We define three further maps
    \begin{itemize}
        \item $\Delta:C\to C\wedge C$, the diagonal map $\xi\mapsto (\xi, \xi)$;
        \item $\Phi: V \to \SP(V\wedge V)$ mapping $\xi_I\mapsto \sum_{I=A\sqcup B}(\xi_A,\xi_B)$;
        \item $\Psi: V \to \SP(V\wedge V)$ mapping $\xi_I\mapsto \sum_{I=A\cup B}(\xi_A,\xi_B)$.
    \end{itemize}
By abuse of notation, $\Phi$ and $\Psi$ will also denote their natural extensions of the form $\SP(V)\rightarrow \SP(V\wedge V)$. We thus obtain a square of maps \begin{equation}\label{eq:diagdeltasigmaphipsi}
        \begin{tikzcd}
            C\rar["\sigma"] \dar["\Delta"] & \SP(V)\dar["\Phi", shift left =2.5pt]\dar["\Psi"', shift right=2.5pt]\\
            C\wedge C\rar["\sigma\triangle \sigma"] & \SP(V\wedge V),
        \end{tikzcd}
    \end{equation}
which should be thought of as two diagrams, one with $\Phi$ and one with $\Psi$.

\begin{proposition}\label{prop:sigmaisiso}
The cohomology map $\sigma^*:\widetilde{H}^*(V)\to \widetilde{H}^*(C)$ is an isomorphism.    
\end{proposition}
\begin{proof}
    This is analogous to the proof of Theorem 4.1 from \cite{Tillmanthorpe2014}. Write $V_k=\vee_{i=1}^kD_i$ and observe that $\sigma$ restricts to a multivalued function of cofibrations from $C_k\to C_{k+1}\to C_{k+1}/C_k$ to $V_k\to V_{k+1}\to V_{k+1}/V_k$, in both of which the third space is $D_{k+1}$, of the form
    \begin{equation*}
    \begin{tikzcd}
    C_k\rar[hook]\dar["\sigma_k"]&
    C_{k+1}\rar[two heads]\dar["\sigma_{k+1}"] & D_{k+1}\dar[hook,"\iota_{D_{k+1}}"]\\
            \SP(V_k)\rar[hook]& \SP(V_{k+1})\rar[two heads]& \SP(D_{k+1})
        \end{tikzcd}
    \end{equation*} 
    The map $\sigma_1^*$ is the identity map. We induct using the map of long exact sequences
    \begin{equation*}
        \begin{tikzcd}
            ...\rar["\partial^*"]& \widetilde{H}^*(D_k)\rar\dar[equal] & \widetilde{H}^*(V_{k+1})\rar\dar["\sigma_{k+1}^*"] &\widetilde{H}^*(V_k)\rar["\partial^*"]\dar["\sigma_k^*"] & ...\\
            ...\rar["\partial^*"]& \widetilde{H}^*(D_k)\rar & \widetilde{H}^*(C_{k+1})\rar &\widetilde{H}^*(C_k)\rar["\partial^*"] & ...
        \end{tikzcd}
    \end{equation*}
     from Lemma \ref{lem:fstarandLES}, and the 5-lemma, to conclude that $\sigma_k^*$ is an isomorphism for all $k\ge 1$. As $C$ and $V$ are the colimits of the filtrations $\{C_k\}$ and $\{V_k\}$, $\sigma^*$ is an isomorphism by the naturality of cohomology maps.
\end{proof}

\subsection{Commutativity of diagram \ref{eq:diagdeltasigmaphipsi}}
\begin{proposition} The diagram \ref{eq:diagdeltasigmaphipsi} with $\Psi$ commutes, that is $\Psi\circ \sigma=(\sigma\triangle \sigma) \circ \Delta$.
\end{proposition}
\begin{proof}
    By direct computations 
    $$(\Psi\circ \sigma)(\xi_I)=\Psi\Big(\sum_{J\subseteq I} \xi_J\Big)=\sum_{A,B,J\subseteq I:A\cup B=J} (\xi_A,\xi_B)$$
    and 
    $$(\sigma\triangle \sigma) \circ \Delta(\xi_I)=(\sigma\triangle \sigma)(\xi_I,\xi_I)=\sum_{A,B\subseteq I}(\xi_A,\xi_B).$$
    The two summations agree: in the first summation $J$ is fully determined by $A$ and $B$ and can be ignored.
\end{proof}

The remaining of this section deals with the technicalities of proving the following.
\begin{proposition}\label{prop:phiandpsihomotopic}
    If $M_0=\emptyset$ and $X=S^{2l}$ is an even sphere, the maps $\Phi$ and $\Psi$ are based homotopic.
\end{proposition}

For an integer $k\ge 1$, write $[k]=\{1,...,k\}$. Let  $\mathcal{I}=\{I_i: i\in [k]\}$ be a sequence of finite sets. The \textit{coloured configuration space on $M$ of type $\mathcal{I}$} is the space $$C_{\mathcal{I}}(M):=F_{\sqcup_{i\in [k]}I_i}(M)/\prod_{i\in [k]}\Sym(I_i),$$
where $F_S(M)$ is the ordered configuration on $M$ of points labeled by $S$, and the product of symmetric groups acts $\prod_{i\in [k]}\Sym(I_i)$ acts by permuting the corresponding co-ordinates. We may make a similar definition $$D_{\mathcal{I}}(M,X)=F_{\sqcup_{i\in [k]}I_i}(M)\times_{\prod_{i\in [k]}\Sym(I_i)} X^{\sqcup_{i\in [k]}I_i}/\sim,$$ where $\sim$ collapses to one point, called $\infty$, all configurations with some $x_i=x_0$.  
\begin{lemma}\label{lem:configurationmapnullhomotopic}
    Let $l>0$, and $A,B,C$ be three disjoint sets. If $A$ is non-empty, then the map 
    \begin{align*}
        \mu:D_{A, B, C}(M,S^{2l})&\longrightarrow D_{A\cup B}(M,S^{2l})\wedge D_{A\cup C}(M,S^{2l}),\\
        (\xi_A,\xi_B, \xi_C)&\longmapsto (\xi_A\cup \xi_B, \xi_A\cup \xi_C)
    \end{align*} is based-null-homotopic.
\end{lemma}
\begin{proof}
   For brevity, we write $D_{\mathcal{I}}=D_{\mathcal{I}}(M,S^{2l})$. As observed in \cite{BodigheimerRationalCohoConfSpacesSurfaces}, $D_{\mathcal{I}}$ given any type $\mathcal{I}$,   is homeomorphic to the Thom space of the vector bundle
    \begin{equation*}
         \eta_{\mathcal{I}}:E_{\mathcal{I}}:= F_{\mathcal{I}}\times_{\prod_{i\in [k]}\Sym(I_i)}\R^{2l}\otimes \left(\R\langle \sqcup_{i\in [k]}I_i\rangle\right) \longrightarrow
            C_{\mathcal{I}}. 
    \end{equation*}
    Under these homeomorphisms, the map $\mu$ is identified with the Thomification, $\Th(\widetilde{f})$, of the bundle morphism
    \begin{center}
        \begin{tikzcd}
            E_{A, B, C}\dar["\eta_{A,B,C}"]\rar["\widetilde{f}"] &
            E_{A\sqcup B }\times E_{A \sqcup C}\dar["\eta_{A\sqcup B}\times \eta_{A\sqcup C}"] 
            \\
            C_{A,B,C}\rar["f"] & C_{A\sqcup B}\times C_{A\sqcup C},
        \end{tikzcd}
    \end{center}
    which on $F_{\mathcal{I}}\times \R^{2l}\otimes \R\langle \sqcup_{i\in [k]}I_i\rangle$ is defined as the product of the map
    $$(x_A,x_B,x_C)\longmapsto (x_A \cup x_B,x_A \cup x_C),$$
    the identity on $\R^{2l}$ and the map
    \begin{align*}
        \R\langle A\sqcup B \sqcup C\rangle&\longrightarrow \R\langle A\sqcup B\rangle\oplus \R\langle A\sqcup C\rangle\\
        (a\oplus b\oplus c)&\longmapsto  \big((a\oplus b)\oplus (a\oplus c)\big),
    \end{align*}
    and descends after quotienting by the diagonal action. 
    Observe that $\widetilde{f}$ is an injective linear map on the fibres, which is a necessary condition for a morphism of vector bundles to give a map of Thom spaces. 

   We now factor $(f,\widetilde{f})$ as the composition of two bundle maps both injective on the fibres. First, we include $\eta_{A,B,C}$ into its stabilisation $\eta_{A,B,C}\oplus \underline{\R}$. Secondly, observe that for any finite set $I$, the vector space $\R\langle I\rangle$ contains the $\Sym(I)$-invariant vector $v_I:=\sum_{i\in I}i$. By also picking an arbitrary non-zero vector $v\in \R^{2l}$, we obtain the nowhere vanishing section $s$ of $\eta_{A\cup B}\times \eta_{A\cup C}$ given by $p\longmapsto (p,v\otimes (v_{A\cup B}\oplus (-v_{A\cup C})))$. Furthermore, on every fibre, $s$ does not lie in the image of $\widetilde{f}$; this can be checked on the direct summand $\R\langle A\rangle$ of each of $\R\langle A\sqcup B\rangle$ and $\R\langle A\sqcup C\rangle$: $\widetilde{f}$ takes the same values in both copies $\R\langle A\rangle$, whereas $s$ takes the same values with opposite signs. In summary, we have obtained the diagram of bundle maps
    \begin{center}
        \begin{tikzcd}
            E_{A,B,C}\rar[hook] \dar & E_{A,B,C}\oplus \underline{\R}\dar \rar["\widetilde{f}\oplus s"] & E_{A\cup B}\times E_{A\cup C}\dar \\
            C_{A,B,C}\rar[equal] & C_{A,B,C}\rar["f"] & C_{A\cup B}\times C_{A\cup C}\uar[bend left, "s"]
        \end{tikzcd}
    \end{center}
    where both squares are injective on fibres. We conclude that the map $\mu$ factors through
    \begin{equation*}
        \Th(E_{A,B,C})\to \Th(E_{A,B,C}\oplus \underline{\R})\cong \Th(E_{A,B,C})\wedge S^1\to \Th( E_{A\sqcup B })\wedge\Th(E_{A \sqcup C}).
    \end{equation*}
    The canonical inclusion of $\Th(E_{A,B,C})$ in its reduced suspension is based-null-homotopic, and thus so is $\mu$.
\end{proof}

\begin{proof}[Proof of Proposition \ref{prop:phiandpsihomotopic}]
    By the definition $V=\vee_{p\ge 1}D_p$, we may write the smash $V\wedge V$ as the wedge $\vee_{q,r\ge 1}D_q\wedge D_r$, so that $\SP(V\wedge V)=\wedge_{q,r\ge 1}(D_q\wedge D_r)$. As a result, the maps $\Phi,\Psi:V\to \SP(V\wedge V)$ decompose to $(p,q,r)$-components  $$\Phi^{(p,q,r)},\Psi^{(p,q,r)}: D_p\longrightarrow \SP(D_q\wedge D_r).$$
    We will prove that these component maps are pairwise base-homotopic.
    
    To begin with, the component $\Phi^{(p,q,r)}$ is given explicitly by  $$\Phi^{(p,q,r)}(\xi_I)=\sum_{\substack{A,B:|A|=q,\\|B|=r,\\ A\sqcup B=I}}(\xi_A,\xi_B),$$ where $|I|=p$. The summation clearly vanishes if $p\neq q+r$. The component $\Psi^{(p,q,r)}$ has the analogous formula 
    $$\Psi^{(p,q,r)}(\xi_I)=\sum_{\substack{A,B:|A|=q,\\|B|=r, \\A\cup B=I}}(\xi_A,\xi_B),$$ which  only differs from $\Phi^{(p,q,r)}$ in that the union $A\cup B$ is no longer disjoint. This too clearly vanishes if $q+r<p$ and, for $p=q+r$, we have $\Phi^{(p,q,r)}=\Psi^{(p,q,r)}$. Thus for all $p\le q+r$, the components $\Phi^{(p,q,r)}$ and $\Psi^{(p,q,r)}$ are in fact equal on the nose. 
    
    For $p>q+r$, $\Psi^{(p,q,r)}$ factors as the composition
    \begin{equation*}
        D_p\longrightarrow \SP(D_{q+r-p, p-r,p-q})\longrightarrow \SP(D_q\wedge D_r)
    \end{equation*}
    by first partitioning the set $I$ of size $p$ into three disjoint sets $C,A',B'$ of sizes $q+r-p, p-r, p-q$, respectively, and then mapping $\xi_{C,A',B'}$ to $\xi_{C\cup A'}\wedge \xi_{C\cup B'}$. This latter map is $\mu$ from Lemma \ref{lem:configurationmapnullhomotopic} and is based-null homotopic. Thus $\Psi^{(p,q,r)}$ is based-homotopic to the trivial map which is equal to $\Phi^{(p,q,r)}$.
\end{proof}

\begin{corollary}\label{cor:DeltacommuteswithPhi}
The diagram
    \begin{center}
        \begin{tikzcd}
            
            \widetilde{H}^*(V\wedge V) \rar["(\sigma\triangle \sigma)^*"]\dar["\Delta^*"] & \widetilde{H}^*(C\wedge C)\dar["\Phi^*"]\\
            \widetilde{H}^*(V)\rar["\sigma^*"] & \widetilde{H}^*(C).
        \end{tikzcd}
    \end{center}
    commutes.
\end{corollary}
    
\section{The superposition product and its relation with $\Phi$}\label{sec:superpositionprodefinition}
Let $C_{n,m}(M)$ be the coloured configuration space of type $\{[n],[m]\}$, 
and $D_{n,m}:=D_{n,m}(M,S^{2d})$ its labelled version. 
The space $C_{n,m}(M)$ is both an open subset of the product $C_n\times C_m$, and a degree
${{n+m} \choose n}$ 
covering map of $C_{n+m}$, giving the maps
\begin{center}
    \begin{tikzcd}
        C_n\times C_m & C_{n,m}\lar[hook,"i"']\rar["p"] & C_{n+m}
    \end{tikzcd}
\end{center}
Using the cohomology transfer map $p_!$, we obtain the \textit{superposition product} $${\sup}_{n,m}=p_{!}\circ i^{*}: H^*(C_n\times C_m)\to H^*(C_{n+m}).$$
We will suppress the index of $\sup$.

We recall that $D_n$ is a Thom space over $C_n$ and $D_n\wedge D_m$ is also a Thom space over $C_n\times C_m$. Further recall that the map $\Phi$ restricts to a multivalued function $\Phi^{(n+m,n,m)}:D_{n+m}\to \SP(D_n\wedge D_m)$. We will henceforth suppress the superscript of $\Phi$.
\begin{proposition}\label{Prop:supcommuteswithPhi}
    The diagram 
    \begin{center}
        \begin{tikzcd}
            H^*(C_n\times C_m)\rar["\sup"]\dar["Th","\cong"']& H^*(C_{n+m})\dar["Th","\cong"'] \\
            \widetilde{H}^*(D_n\wedge D_m)\rar["\Phi^*"] & \widetilde{H}^*(D_{n+m}),
        \end{tikzcd}
    \end{center}
    commutes. (The degree shifts of the Thom isomorphisms are kept implicit.)
\end{proposition}
\begin{proof}
     Observe that $D_{n,m}$ is also the Thom space of an oriented vector bundle over $C_{n,m}$. Defining the maps $\widetilde{i}$ and $\widetilde{p}$ in the obvious way, we get two pullback squares of bundles
     \begin{center}
    \begin{tikzcd}
    D_n\wedge D_m-\{\infty\}\dar & D_{n,m}-\{\infty\} \lar[hook,"\widetilde{i}"]\rar["\widetilde{p}"]\dar & D_{n+m}-\{\infty\}\dar \\
        C_n\times C_m & C_{n,m}\lar[hook,"i"]\rar["p"] & C_{n+m}.
    \end{tikzcd}
\end{center}
Therefore, the Thom isomorphisms commute with the compactified maps $\widetilde{i}^*_{\infty}$
and $\widetilde{p}^*_{\infty}$. Moreover, from Lemma \ref{lem:thomandtransfer} proven below, as $p$ is a covering space, there are transfer maps $p_!$ and $(\widetilde{p}_{\infty})_!:\widetilde{H}^*(D_{n,m})\to \widetilde{H}^*(D_{n+m})$, which also commute with the Thom isomorphisms.

Now ${p}_!\circ {i}^*$ is by definition the superposition product. But also $(\widetilde{p}_{\infty})_!\circ \widetilde{i}^*_{\infty}$ equals $\Phi^*$, because $\Phi^{(n+m,n,m)}$ factors as the composition $D_{n+m}\to \SP(D_{n,m})\to\SP( D_n\wedge D_m)$, where the first map is the inverse of the covering map $\widetilde{p}_{\infty}$ and the second map is $\SP(\widetilde{i}_{\infty})$. In cohomology, we obtain the equality $(\widetilde{p}_{\infty})_!\circ \widetilde{i}^*_{\infty}=\Phi^*$, implying the desired diagram commutes.
\end{proof}

\begin{lemma}\label{lem:thomandtransfer}
    Let $\pi:E\to B$ be an oriented rank $k$ vector bundle, and $p:\widetilde{B}\to B$ a degree $n$ covering space. Then the multivalued function $\Th(E)\to \SP(\Th(p^*E))$ inverse to the covering map $\bar{p}:p^*E\to E$ induces a transfer map $\bar{p}_!:\widetilde{H}^*(\Th(p^*E))\to \widetilde{H}^*(\Th(E))$ which commutes with the transfer map $p_!$, so that the diagram
    \begin{center}
        \begin{tikzcd}
            \widetilde{H}^{*+k}(\Th(p^*E))\rar["\bar{p}_!"] & \widetilde{H}^{*+k}(\Th(E))\\
            {H}^*(\widetilde{B})\uar["\Th","\cong"'] \rar["p_!"] & H^*(B)\uar["\Th"]
        \end{tikzcd}
    \end{center}
    commutes.
\end{lemma}
\begin{proof}
    For both bundles $F=E, p^*E$, we replace the pair $(\Th(F),\infty)$ by the pair $(\bbD(F),\bbS(F))$ which is isomorphic by a radial homeomorphism, so that $\bar{p}:(\bbD(p^*E),\bbS(p^*E))\to (\bbD(E),\bbS(E))$ is a degree $n$ covering space of pairs. Under the excision isomorphisms $\bar{p}_!:H^*(\bbD(p^*E),\bbS(p^*E))\to H^*(\bbD(E),\bbS(E))$ is equal to the transfer map defined in the classical way, namely on singular cochains, it is the dual of the chain map: \begin{align*}
        \bar{p}^!:C_*(\bbD(E),\bbS(E))&\longrightarrow C_*(\bbD(p^*E),\bbS(p^*E))\\
        \sigma &\longmapsto \sum_n \widetilde{\sigma}
    \end{align*}
    where the summation is taken over all $n$ $\bar{p}$-preimages of the simplex $\sigma$. The transfer map $p_!$ has a similar definition. We pick a cochain $u_E\in C^k(\bbD(E),\bbS(E))$ representing the Thom class of $E\to B$, so that $p^*(u_E)$ represents the Thom class of $p^*E\to \widetilde{B}$. It suffices to check that the diagram
    \begin{center}
        \begin{tikzcd}
            C^{i+k}(\bbD(E),\bbS(E))\rar["\bar{p}_!"] &   C^{i+k}(\bbD(E),\bbS(E))\\
            C^i(\widetilde{B})\uar["\widetilde{\pi}^*(-)\smile \bar{p}^*u_E"]\rar["p_!"] & C^i(B)\uar["\pi^*(-)\smile u_E"]
        \end{tikzcd}
    \end{center}
    commutes, i.e. that for any class $\alpha$, we have $$\pi^*(p_!(\alpha))\smile u_E=\bar{p}_!(\widetilde{\pi}^*(\alpha)\smile \bar{p}^*u_E).$$
    But from the definition of the cup product we have $\bar{p}_!(\widetilde{\pi}^*(\alpha)\smile \bar{p}^*u_E)=\bar{p}_!(\widetilde{\pi}^*(\alpha))\smile u_E$, and then by evaluating on a simplices, $\bar{p}_!(\widetilde{\pi}^*(\alpha))=\pi^*(p_!(\alpha))$.    
\end{proof}

\section{Proof of Theorem \ref{thm:B}}
\begin{proposition}\label{prop:sigmaismultiplicative}
For any $l\ge 1$, the multivalued function $\sigma$ induces a multiplicative isomorphism $$\sigma^*:\widetilde{H}^*(C(M;S^{2l}))\longrightarrow \bigoplus_{n\ge 1}\widetilde{H}^{*}(C_n(M))[2ln]$$
with the cup-product in the domain and $\sup$ in the target.
\end{proposition}
\begin{proof}
For positive $l$, and a fixed $i$, the abelian group $H^{i-2ln}(C_n(M))\cong \widetilde{H}^i(D_n)$ is non-zero only for finitely many $n$, which justifies the identification
$$\widetilde{H}^i(\vee_{n\ge 1}D_n)\cong \bigoplus_{n\ge 1}\widetilde{H}^i( D_n),$$
with a direct sum as opposed to a direct product. That $\sigma^*$ is an isomorphism was proved in Proposition \ref{prop:sigmaisiso}.

For the multiplicativity, we look at the following diagram (where $C_n$ is shorthand for $C_n(M)$ for economy of space)
   \begin{equation*}\small
        \begin{tikzcd}
            \widetilde{H}^*(C)\otimes \widetilde{H}^*(C)\dar["\times"]\rar["\sigma^*\otimes \sigma^*"]
            & \widetilde{H}^*(\vee_{n\ge 1}D_n)\otimes \widetilde{H}^*(\vee_{n\ge 1}D_n)   \dar["\times"]\rar["\Th\otimes\Th"]
            &  \bigoplus_{n,m\ge 1}{H}^{*}(C_n)\otimes {H}^{*}(C_n)\dar["\times"]
            \\
            \widetilde{H}^*(C\wedge C)\rar["(\sigma\triangle\sigma)^*"]\dar["\Delta^*"]
            &\widetilde{H}^*(\vee_{n\ge 1}D_n\wedge \vee_{n\ge 1}D_n)\rar["\Th"]\dar["\Phi^*"]
            & \bigoplus_{n,m\ge 1}{H}^{*}(C_n\times C_m)\dar["\sup"]
            \\
            \widetilde{H}^*(C)\rar["\sigma^*"]
            & \widetilde{H}^*(\vee_{n\ge 1}D_n)\rar["\Th"]
             & \bigoplus_{n\ge 1}{H}^* (C_n).
        \end{tikzcd}
    \end{equation*}
    Here every square commutes: the top-left from Proposition \ref{prop:tensorandcohomologymapscommute}, the bottom-left from Corollary \ref{cor:DeltacommuteswithPhi}, the top-right because the Thom isomorphism commutes with the cross-product, and the bottom-right by Proposition \ref{Prop:supcommuteswithPhi}. Looking only at the corners of the large square, the leftmost vertical composition is the cup product and the rightmost column is the superposition product, whereas the top-most and bottom maps are the scanning maps. This commutativity proves exactly that $\sigma^*$ respects the products.
\end{proof}

\begin{theorem}\label{thm:main}
There is an isomorphism of graded rings $$s:\left({H}^*(\Gamma_{\partial}(M,S^{2l})),\smile\right)\longrightarrow \Big(\bigoplus_{n\ge 0}H^*(C_n(M))[2nl],\sup\Big).$$
\end{theorem}
\begin{proof}
    The restricted map $s:\widetilde{H}^*(\Gamma_{\partial}(M,S^{2l}))\to \bigoplus_{n\ge 1}H^*(C_n(M))[2nl]$ is an isomorphism of non-unital rings by combining Proposition 3.6 and Corollary 4.2 of \cite{Tillmanthorpe2014}, with our Proposition \ref{prop:sigmaismultiplicative}. (That the target is a ring, i.e. that $\sup$ is associative and graded commutative, follows \textit{a posteriori} since $\smile$ has these properties.) Passing to the unreduced cohomology ${H}^*(\Gamma_{\partial}(M,S^{2l}))$ adds back the unit to the cohomology ring. Similarly, the summand $H^*(C_0(M))[0]\cong \Z[0]$ can be checked to be generated by the unit for $\sup$, making the extended isomorphism $s$ an isomorphism of unital rings. 
\end{proof}

\section{An alternative definition of the superposition product}\label{sec:alternativedefinition}
We give an alternative definition of the superposition product appearing in \cite{bianchi2021mapping,Moriyama, randalwilliams2023configuration}. Let $C_n(M)^{\infty}$ be the one point compactification of $C_n(M)$, where the point $\infty$ corresponds to configurations where two points collide or one point goes to the boundary. It is straightforward to define a superposition map 
\begin{align*}
        \mu_{n,m}:C_n(M)^{\infty}\wedge C_m(M)^{\infty}&\longrightarrow  C_{n+m}(M)^{\infty}\\
         (s,t)&\longmapsto \begin{cases}
        \infty \text{ if } s\cap t\neq \emptyset \text{ or } s=\infty \text{ or } t=\infty, \\
        s\cup t \text{ otherwise.}
        \end{cases} 
    \end{align*}

From now on assume that $M$ is orientable. We observe that $C_n(M)^{\infty}\wedge C_m(M)^{\infty}=(C_n(M)\times C_n(M))^\infty$. Using the Poincare--Lefschetz duality isomorphism $\PL$ and the homology map $(\mu_{n,m})_*$, we define $\widetilde{\sup}_{n,m}$ so that the diagram
\begin{equation*}
    \begin{tikzcd}
        H^*(C_n(M)\times C_m(M))\rar[dashed, "\widetilde{\sup}_{n,m}"]\dar["\cong","\PL"'] & H^*(C_{n+m}(M))\dar["\cong","\PL"']\\
        \widetilde{H}_{d(n+m)-*}(C_n(M)^{\infty}\wedge C_m(M)^{\infty})\rar["(\mu_{n,m})_*"]& \widetilde{H}_{d(n+m)-*}(C_{n+m}(M)^{\infty})
    \end{tikzcd}
\end{equation*}
commutes. 

\begin{proposition}\label{prop:equivalentdefinitions}
    The map $\widetilde{\sup}_{n,m}$ coincides with $\sup_{n,m}$ from Section \ref{sec:superpositionprodefinition}.
\end{proposition}
\begin{proof}
    Open inclusions give opposite maps after compactifying, thus $\mu_{n,m}$ factors as
\begin{center}
    \begin{tikzcd}
        C_n(m)^{\infty}\wedge C_m(m)^{\infty}\rar[hook,"i^{\infty}"] & C_{n,m}(M)^{\infty}\rar["p^{\infty}"] & C_{n+m}(M)^{\infty}.
    \end{tikzcd}
\end{center}
Consequently, we may factor $$\widetilde{\sup}_{n,m}=(\PL^{-1}\circ p^{\infty}_* \circ \PL)\circ(\PL^{-1}\circ i^{\infty}_* \circ \PL),$$
using the $\PL$ isomoprhism for $C_{n,m}(M)$. By standard properties of Poincar\'e--Lefschetz duality $(\PL^{-1}\circ i^{\infty}_* \circ \PL)$ is equal to the cohomology map $i^*$, whereas the Gysin map $(\PL^{-1}\circ p^{\infty}_* \circ \PL)$ is equal to the transfer map $p_!$. As a result $\widetilde{\sup}_{n,m}=p_!\circ i^*=\sup_{n,m}$.
\end{proof}
    Proposition \ref{prop:equivalentdefinitions} can also be proved for non-orientable manifolds, but care must be taken in using the orientation sheaf for Poincar\'e--Lefschetz duality.
\bibliographystyle{amsalpha}

\bibliography{biblio}

\end{document}